\documentclass[12pt]{article}
\usepackage{amsmath,amsfonts,amssymb,amsthm,
mathrsfs}

\theoremstyle{definition}

\newtheorem{remark}{Remark}

\theoremstyle{plain}
\newtheorem{theorem}{Theorem}

\newcommand{\E}{\mathsf{E}}

\newcommand{\Prob}{\mathsf{P}}

\begin{document}

\title{On some extensions of generalized counting processes}
\author{
Lyudmyla Sakhno
\\
{\it\small Faculty of Mechanics and Mathematics,
	Taras Shevchenko National University of Kyiv,  Ukraine}\\
Artem Storozhuk\\
{\it\small Faculty of Mechanics and Mathematics,
	Taras Shevchenko National University of Kyiv, Ukraine}}

\maketitle

\begin{abstract} We study different fractional extensions of the Poisson process and generalized counting processes by introducing time-change represented by the inverse to the sums of stable and tempered stable subordinators. We state the governing equations for probability distributions and probability generating functions which involve fractional derivatives of different orders. Closed form expressions for probability distributions and probability generating functions are also provided for several considered models.

\medskip

\medskip
\end{abstract}

\bigskip

\section{Introduction}

Numerous recent studies have been devoted to fractional extensions of stochastic processes defined by introducing suitable fractional derivatives into the equations to which these processes pertain. On the other side, the connection of stochastic processes with fractional equations is underpinned via Bochner-type subordination by means of inverse stable subordinators. The role of the Mittag-Leffler function should he highlighted here as that one which gives the Laplace transform of the inverse stable subordinator and at the same time presents the eigenfunction of the classical fractional Caputo-Djrbashian derivative. The literature on the topic can be traced back to 1990-s and even earlier, we refere here to several more recent sources relevant to our consideration: \cite{BO2009,BO2010, DovOrsToaldo, GKL202, MS, MNV, OBPTRF, OT, OT2017}, among many others. To go beyond the models of fractional Poisson processes, in \cite{DiC} the generalised fractional counting processes were introduced and studied, followed by their further extensions in  various directions   (see, for example, \cite{KLS}, \cite{KK2022a}, \cite{KK2022b}, \cite{KK2022Arxiv},  and references therein)

Generalized fractional calculus introduced in \cite{K}, \cite{T} has inspired the intensive studies of new types of equations and stochastic processes. In particular, new models of Poisson and generalized counting processes governed by the equations with the generalized fractional convolution-type derivatives  were introduced and investigated,  e.g., in \cite{BS2019}, \cite{BS2024}, \cite{KK2022a}, \cite{KK2022Arxiv}.

The convolution-type derivatives allow to study  the properties of subordinators and their inverses in the unifying manner (\cite{T, MT}). In particular,  the densities of inverse subordinators and their Laplace transforms solve the equations  given in terms of   convolution-type derivatives. To be  more precise, Laplace transforms of inverse subordinators are shown to be eigenfunctions of the corresponding  convolution-type derivatives, although an explicit expression for the Laplace transform is not at our disposal in a general case. However,  these properties provide important tools for study of time-changed processes:  from the equations for the densities of inverse subordinators and their Laplace transforms one can deduce the  equations for probabilities of  processes time-changed by inverse subordinators and equations for some related functionals (see, e.g., \cite{BS2019}, \cite{BS2024}). 

Our paper was greatly motivated  by the papers \cite{OT2017}, \cite{DovOrsToaldo} and \cite{BO2010}. We aim to study different fractional extensions of the Poisson process and generalized counting processes by introducing time-change represented by the inverse to the sums of stable and tempered stable subordinators.

The paper is organized as follows. 

In {\it Section 2} we collect some definitions and facts needed for further reference and use.

 In {\it Section 3} we introduce and study the Poisson and generalized counting processes time-changed by the inverse to a sums of stable subordinators. We  present the governing equations for the probability distributions, which  involve the fractional derivatives of different orders, in a form of a generalized telegraph-type operator in time. In particular cases we are able to write explicit expressions for probability distributions and also expressions for probability generating functions, as consequences of available expressions for the Laplace transforms of inverse processes. The formulas are based heavily on the use of Mittag-Leffler functions.  We next study in {\it Section 4} the processes time-changed by the inverse to a sum of tempered stable subordinators and their governing equations. Finally, in {\it Section 5},  we consider a problem which is not related to the time-change of counting processes, but concerned with their applications, namely, with evaluation of the non-ruin probability for the risk models based on generalized counting processes. We present the expression for non-ruin probability in terms of the Mittag-Leffler functions, extending the existing results.

\section{Preliminaries}

In this section we collect the necessary definitions and facts, in particular, on generalized counting processes and their fractional extensions, as prepequisites for our further study.

\subsection{Generalized fractional derivatives} 
We review briefly the main definitions and some  facts on the generalized fractional derivatives (for more details see, e.g., \cite{MT, T}.) 

Let $f(x)$ be a Bern\v{s}tein function:
\begin{equation}\label{fx}
f(x)=a+bx+\int_{0}^{\infty}\left(1-e^{-xs}\right)\overline{\nu}_f(ds), \,\,x>0, \,\,a, b\geq0, 
\end{equation}
with L\'evy measure $\overline{\nu}_f(ds)$ such that
$
\int_{0}^{\infty}\left(s\wedge 1\right)\overline{\nu}_f(ds)<\infty.$

The generalized Caputo-Djrbashian (C-D) derivative, or convolution-type deri\-va\-tive, with respect to the Bern\v{s}tein function $f$ is defined on the space of absolutely continuous functions as follows (\cite{T}, Definition 2.4):
\begin{equation}\label{fDt}
\mathcal{D}^f_t u(t)=b \frac{d}{dt}u(t)+\int_{0}^{t}\frac{\partial }{\partial t }u(t-s)\nu_f (s)ds,
\end{equation}
where $\nu_f (s)=a+\overline{\nu}_f(s,\infty)$ is the tail of the L\'evy measure $\overline{\nu}_f(s)$ of the function~$f$.

In the case where $f(x)=x^{\alpha},x>0, \alpha\in (0,1)$, the derivative \eqref{fDt} reduces to the classical fractional C-D
derivative:
\begin{equation}\label{C-D}
\mathcal{D}^f_t u(t)=\frac{d^{\alpha}}{dt^{\alpha}}u(t)=\frac{1}{\Gamma(1-\alpha)}\int_{0}^{t}\frac{u^\prime(s)}{\left(t-s\right)^{\alpha}}ds.
\end{equation}
For the Laplace transform of the derivative \eqref{fDt} the following relation holds (\cite{T}, Lemma 2.5):
\begin{equation*}\label{LfDt}
\mathcal{L}\left[\mathcal{D}^f_t u\right](s)=f(s)\mathcal{L}\left[u\right](s)-\frac{f(s)}{s}u(0), s>s_0,
\end{equation*}
for $u$ such that $|u(t)|\leq\mathsf{M} e^{s_0t}$, $M$ and $s_0$ are some constants. Similarly to the C-D fractional derivative, the convolution type derivative can be  alternatively defined via its Laplace transform.

The generalization of the classical Riemann-Liouville (R-L) fractional derivative is introduced in \cite{T} by means of another convolution-type derivative with respect to $f$ given as
\begin{equation}\label{fDDt}
\mathbb{D}_t^f u(t)=b \frac{d}{dt}u(t)+\frac{d}{dt}\int_{0}^{t}u(t-s)\nu_f (s)ds.
\end{equation}
The derivatives $\mathcal{D}^f_t$ and $\mathbb{D}_t^f$ are related as follows (see, \cite{T}, Proposition 2.7):
\begin{equation}\label{fDtrelation}
\mathbb{D}_t^f u(t)=\mathcal{D}^f_t u(t) +\nu_f (t)u(0).
\end{equation}

Let $H^f(t)$, $t\geq0,$ be a subordinator, that is, nondecreasing L\'evy process with  Laplace transform 
$$\mathcal{L}[H^f(t)](s)=\E e^{-sH^f(t)}= e^{-t f(s)},$$
 where the function $f$, called the Laplace exponent, is a Bern\v{s}tein function. Let $Y^f$ be the inverse process defined as
\begin{equation}\label{Yf}
Y^f(t)=\inf \left\{s\geq0:H^f(s)>t\right\}.
\end{equation}
It was shown in \cite{T} that the distribution of the inverse process $Y^f$ has a density 
$\ell_f(t,x)={\rm P}\{Y^f(t)\in dx\}/dx$ 
provided that the following condition holds:

\medskip
{\it Condition I.} $\overline{\nu}_f(0,\infty)=\infty$ and the tail $\nu_f(s)=a+\overline{\nu}_f(s,\infty)$ is absolutely continuous.
\medskip

 The Laplace transform of the density of the inverse subordinator with respect to $t$ is (\cite{T}):
\begin{equation}\label{t-lapl-inv}
\mathcal{L}_t\left(\ell_f(t,x)\right)(r)=\frac{f(r)}{r}e^{-xf(r)}.
\end{equation}

The density $\ell_f(t,u)$ of the inverse process $Y^f$ satisfies the follo\-wing equation (\cite{T}):
\begin{align}\label{dlf}
	&{\mathbb{D}}_t^f \ell_f(t,u)=-\frac{\partial}{\partial  u} \ell_f(t,u), \, t>0,
	\,\,\,0<u<\infty,\, {\rm if}\, b=0,\, 
	\, 0<u<t/b,\, {\rm if}\, b>0,
\end{align}
subject to
\begin{equation}\label{dlfin}
\ell_f(t, {u}/{b})=0,\,\, \ell_f(t,0)=\nu_f(t),\,\, \ell_f(0,u)=\delta(u).
\end{equation}

The space Laplace transform of the density $\ell_f(t,x)$ 
\begin{equation}\label{tilde_l}
\tilde{\ell_f}(t,\lambda)=\int_{0}^{\infty}e^{-\lambda x} \ell_f(t,x)dx= {\E}e^{-\lambda Y^f(t)}, \,\, t\ge0, \,\lambda>0,
\end{equation}
is an eigenfunction of the operator $\mathcal{D}^f_t$, that is, satisfies the equation
	\begin{equation}\label{lapl}
		\mathcal{D}^f_t\tilde{\ell_f}(t,\lambda)=-\lambda \tilde{\ell_f}(t,\lambda),\,\,t>0,
	\end{equation}
with $\tilde{\ell_f}(0,\lambda)=1$ (see, \cite{BS2019, K, MT}). 
If $f(x)=x^{\alpha},x>0, \alpha\in (0,1)$, then $\tilde{\ell_f}(t,\lambda)=
{E}_\alpha(-\lambda t^\alpha)$, where $
{E}_\alpha(\cdot)$ is the Mittag-Leffler function:
\begin{equation}\label{ML}
{E}_\alpha(z)=\sum_{k=0}^\infty \frac{z^k}{\Gamma(\alpha k +1)}, \,\, z\in \mathcal{C}, \, \alpha\in  (0, \infty). 
\end{equation}

\subsection{Mittag-Leffler functions}
In addition to the function \eqref{ML}, the  two-parameter and three parameter Mittag-Leffler functions are fundamental in the study of probability laws of fractional processes, which are defined as
\begin{equation}\label{ML2}
	E_{\alpha,\beta}(z) = \sum_{r=0}^{\infty} \frac{z^r}{\Gamma(\alpha r + \beta)}, \quad \alpha, \beta \in \mathbb{C}, Re(\alpha), Re(\beta) > 0, z \in \mathbb{R};
\end{equation}
\begin{equation}\label{ML3}
	E_{\alpha,\beta}^{\gamma}(z) = \sum_{r=0}^{\infty} \frac{(\gamma)_r z^r}{r! \Gamma(\alpha r + \beta)}, \quad \alpha, \beta, \gamma \in \mathbb{C}, Re(\alpha), Re(\beta), Re(\gamma) > 0,
\end{equation}
where $(\gamma)_r = \gamma(\gamma+1)\dots(\gamma+r-1)$ for $r=1,2,\dots$, and $(\gamma)_0 = 1$ denotes the Pochhammer symbol (see \cite{HMS}).

We will also use the  generalized multivariate Mittag-Leffler function defined as follows
\begin{equation} \label{multML}
	E_{\alpha,\beta}^{\gamma_1, \dots, \gamma_m}(z_1, \dots, z_m) = \sum_{k_1=0}^{\infty} \dots \sum_{k_m=0}^{\infty} \frac{(\gamma_1)_{k_1} \dots (\gamma_m)_{k_m}}{\Gamma\big(\alpha \sum_{j=1}^m k_j + \beta\big)} \frac{z_1^{k_1}}{k_1!} \dots \frac{z_m^{k_m}}{k_m!},
\end{equation}
where $\gamma_1, \dots, \gamma_m \in \mathbb{C}$, $z_1, \dots, z_m \in \mathbb{C}$,  with $Re\gamma_1, \dots, Re\gamma_m >0$, $Re\alpha>0$. Function \eqref{multML} is a particular variant of the multivariate Mittag-Leffler function introduced in \cite{Saxena}. The following Laplace transform formula holds (see, e.g., \cite{CO}):
\begin{equation}\label{multMLfLapl}
		\int_0^\infty e^{-\mu x} x^{\delta-1} E_{\nu,\delta}^{(\gamma_1, \dots, \gamma_M)}(x \eta) dx = \frac{\mu^{\nu \sum_{j=1}^M \gamma_j - \delta}}{\prod_{j=1}^M (\mu^\nu - \eta_j)^{\gamma_j}},
		\end{equation}
for $\eta_1, \dots, \eta_m \in \mathbb{C}$, $\nu>0$, $\left|\frac{\eta_j}{\mu^\nu}\right|<1$.


\subsection{Fractional Poisson process}\label{fractionalPoisson}
The fractional Poisson process (more precisely, time-fractional), denoted by $\{N_{\lambda}^{\nu}(t); t \ge 0\}$ for $\nu \in (0,1]$ and $\lambda >0$, has the  probabilities $p_n(t)=\mathbb{P}\{N_{\lambda}^{\nu}(t)=n\}$ governed by the fractional equations
\begin{equation*}
	\frac{d^\nu p_n(t)}{dt^\nu} = -\lambda(p_n(t) - p_{n-1}(t)), \quad n \ge 0,
\end{equation*}
with $p_{-1}(t) = 0$, subject to the initial conditions
	$p_n(0) = \delta_{n0}$,
which can be explicitly given as
\begin{equation*}
	p_n(t) = (\lambda t^\nu)^n E_{\nu, \nu n+1}^{n+1}(-\lambda t^\nu), \quad n \ge 0,\,\, t > 0.
\end{equation*}


\subsection{Generalized fractional counting processes}
Generalized counting process (GCP) $M(t)$, $t\ge 0$,  was introduced in \cite{DiC} as a counting process defined by following rules:
\begin{enumerate}
	\item $M(0) = 0 \quad \text{a.s.};$
	\item $M(t)$ has stationary and independent increments;
	\item $\mathbb{P}\{M(h) = j\} = \lambda_j h + o(h), \quad \text{for } j = 1, 2, \dots, k;$
	\item $\mathbb{P}\{M(h) > k\} = o(h),$
\end{enumerate}
where $k \in \mathbb{N} \equiv \{1, 2, \dots\}$ is fixed, and $\lambda_1, \lambda_2, \dots, \lambda_k > 0$.

 The probabilities $\tilde p_n(t)=\Prob\left\{M(t)=n\right\}$  depend on $k$ parameters $\lambda_1, \ldots, \lambda_k$ and are given by the formula
\begin{equation*}
             \tilde p_n(t)=\sum_{\Omega(k,n)}^{}\prod_{j=1}^{k}\frac{\left(\lambda_j t\right)^{x_j}}{x_j!}e^{-\Lambda t}, n\geq0,\label{pnt}
\end{equation*}
where $\Omega(k,n)=\left\{(x_1,\dots,x_k)\,:\,\sum_{j=1}^{k}jx_j=n,x_j\in N_0\right\}
$, $\Lambda=\sum_{j=1}^{k}\lambda_j$. \\
GCP performs $k$ kinds of jumps of amplitude $1, 2, \ldots, k$ with rates $\lambda_1, \ldots, \lambda_k$.
Note that GCP comprises as particular cases such important for applications models as the Poisson process of order $k$ and P\'olya-Aeppli process of order $k$ (see, e.g. \cite{KK2022b}, \cite{KLS}).

The probabilities $\tilde p_n(t)$ satisfy
\begin{equation}\label{pGCP}
            \frac{d\tilde p_n(t)}{dt}=- \Lambda \tilde p_n(t) + \sum_{j=1}^{\min\{n,k\}}\lambda_j \tilde p_{n-j}(t), \, n\geq0,
\end{equation}
with the usual initial condition. 
The probabilities $\tilde p_n(t)$ can be also written as 
\begin{equation*}\label{pnt31}
             \tilde p_n(t)=\sum_{\Omega(k,n)}^{}\prod_{j=1}^{k}\frac{\lambda_j^{x_j} }{x_j!}\left(-\partial_{\Lambda}\right)^{z_k} e^{-\Lambda t}, \, n\geq0,
\end{equation*}
where 
             $z_k=\sum_{j=1}^{k}x_j$ (see \cite{KK2022Arxiv}).
             
The probability generating function of GCP is given by (\cite{KK2022a}):
\begin{equation}
             \tilde G(u,t)=\E u^{M(t)}=\exp\Big\{
-\sum_{j=1}^{k}\lambda_j(1-u^j)t\Big\},\, |u|<1.
\end{equation}


Fractional extension of the generalized counting process $M(t)$ was introduced in \cite{DiC} as the process $M^\nu(t)$, $t \ge 0$, $\nu \in (0,1]$, whose probability distribution
$$
	\tilde p_n^\nu(t) = \mathbb{P}\{M^\nu(t) = n\}, \quad n\ge0,$$	
satisfies the following system of fractional difference-differential equations 
\begin{align}\label{Dic1}
	&\frac{\mathrm{d}^\nu \tilde p_n^\nu(t)}{\mathrm{d}t^\nu} = - \Lambda \tilde p_n^\nu(t) + \sum_{r=1}^{\min\{k,n\}} \lambda_r \tilde p_{n-r}^\nu(t),\,\, n>0,\\
	&\frac{\mathrm{d}^\nu \tilde p_0^\nu(t)}{\mathrm{d}t^\nu} = -\Lambda \tilde p_0^\nu(t), \nonumber
\end{align}
for $\Lambda = \lambda_1 + \lambda_2 + \dots + \lambda_k$,  with the initial condition
\begin{equation}\label{Dic2}
	\tilde p_n(0) = \begin{cases}
		1, & n = 0 \\
		0, & n \ge 1.
	\end{cases}
\end{equation}
The next two theorems from \cite{DiC} give the relation of the process $M^\nu(t)$ with the fractional Poisson process  the $N_\Lambda^\nu(t)$ and the expressions for  probabilities $\tilde p_n^\nu(t)$.
\begin{theorem}[\cite{DiC}]\label{thmDic}
	For all fixed $\nu \in (0, 1]$ we have
	\begin{equation}
		M^\nu(t) \stackrel{d}{=} \sum_{i=1}^{N_\Lambda^\nu(t)} X_i, \qquad t \ge 0,
	\end{equation}
	where $N_\Lambda^\nu(t)$ is a fractional Poisson process (see Section \ref{fractionalPoisson}), with intensity $\Lambda = \lambda_1 + \lambda_2 + \dots + \lambda_k$, and  $\{X_n : n \ge 1\}$ is a sequence of i.i.d. random variables, independent of $N_\Lambda^\nu(t)$, such that for any $n \in \mathbb{N}$
	\begin{equation}
		\mathbb{P}\{X_n = j\} = \frac{\lambda_j}{\Lambda}, \qquad j = 1, 2, \dots, k
	\end{equation}
	and where both $N_\Lambda^\nu(t)$ and $X_n$ depend on the same parameters $\lambda_1, \lambda_2, \dots, \lambda_k$.
\end{theorem}

\begin{theorem}[\cite{DiC}]
	The solution $\tilde p_j^\nu(t)$ of the Cauchy problem \eqref{Dic1}--\eqref{Dic2}, for $j \in \mathbb{N}_0, \nu \in (0,1]$ and $t \ge 0$, is given by
	\begin{equation}
		\tilde p_j^\nu(t) = \sum_{r=0}^j \sum_{\substack{\alpha_1+\alpha_2+\dots+\alpha_k=r \\ \alpha_1+2\alpha_2+\dots+k\alpha_k=j}} \binom{r}{\alpha_1, \alpha_2, \dots, \alpha_k} \lambda_1^{\alpha_1} \lambda_2^{\alpha_2} \dots \lambda_k^{\alpha_k} t^{r\nu} E_{\nu,r\nu+1}^{r+1}(-\Lambda t^\nu).
	\end{equation}
\end{theorem}

In the paper \cite{BS2024} the authors studied the time-changed process  ${M}^{\psi,f}(t)=M(H^{\psi}(Y^f(t))$, that is, the generalized counting process with double time-change by an independent subordinator $H^{\psi}$ and an inverse subordinator $Y^f$, which are independent of $M$. The probabilities $\tilde{p}_n^{\psi,f}(t)=\Prob\left\{{M}^{\psi,f}(t)=n\right\}$ and the probability generating function of ${M}^{\psi,f}$ are characterized in the following theorem.
\begin{theorem}[\cite{BS2024}]\label{Th5}
   	 The process ${M}^{\psi,f}$ has probability distribution function
   	 \begin{equation}\label{pnpsif}
   	          \tilde{p}_n^{\psi,f}(t)=\sum_{\Omega(k,n)}^{}\prod_{j=1}^{k}\frac{\lambda_j ^{x_j}}{x_j!}\left(-\partial_{\Lambda}\right)^{z_k}\tilde{\ell}_f(t,\psi (\Lambda)), \, n\ge 0,
    	\end{equation}
    	and $\tilde{p}_n^{\psi,f}(t)$ satisfy the following equation
    	\begin{equation}
   		 \mathcal{D}^f_t \tilde{p}_n^{\psi,f}(t)=-\psi(\Lambda)\tilde{p}_n^{\psi,f}(t) -\sum_{m=1}^{n}\sum_{\Omega(k,m)}^{}\psi^{(z_k)}\left(\Lambda\right) \prod_{j=1}^{k}\frac{(-\lambda_j )^{x_j}}{x_j!}\tilde{p}_{n-m}^{\psi,f}(t),\label{Dtfpn}, \, n\ge 0, \,t>0,
  	\end{equation}
   	 with initial conditions 
    \begin{equation*}
         \tilde{p}_n^{\psi,f}(0)=
         \begin{cases}
                1,\quad n=0,\\
                0,\quad n\geq1.\\
         \end{cases}
     \end{equation*}
        The probability generating function of the process ${M}^{\psi,f}$ is of the form
\begin{equation}\label{Gpsif1}
             \tilde G^{\psi,f}(u,t)=\tilde \ell_f\Big(t, \psi\Big(\sum_{j=1}^{k}\lambda_j(1-u^j)\Big)\Big),\, |u|<1,
\end{equation}
and satisfies the equation
\begin{equation}\label{DftG}
        	\mathcal{D}^f_t \tilde G^{\psi,f}(u,t)=-\psi\Big(\sum_{j=1}^{k}\lambda_j(1-u^j)\Big)\tilde G^{\psi,f}(u,t)
    \end{equation}
    with $\tilde G^{\psi,f}(u,0)=1$.
    The derivatives used in \eqref{Dtfpn} and \eqref{DftG} are the C-D convolution-type derivatives defined in \eqref{fDt}.
\end{theorem}
\noindent{\bf Note about notation.} To avoid complicated notations, we will use in the different sections similar notations for similas objects, which will be valid within a particular section.

\section{Sum of stable subordinators and its inverse and \\corresponding time-changed counting processes}

Consider the following sum
	\begin{equation}\label{2.1}
		\mathcal{H}^\nu(t) = H_1^{2\nu}(t) + (2\lambda)^{\frac{1}{\nu}} H_2^\nu(t), \quad t > 0,\,\,\, 0 < \nu \le \frac{1}{2},
	\end{equation}
	where  $H_1^{2\nu}$, $H_2^\nu$  are independent stable subordinators with parameters ${2\nu}$ and $\nu$ correspondingly, $\lambda > 0$. 
	
	Define the inverse $\mathcal{L}^\nu(t), t > 0$, to the process $\mathcal{H}^\nu(t), t > 0$, as follows
	\begin{equation}\label{1.13}
		\mathcal{L}^\nu(t) = \inf \left\{ s > 0 : H_1^{2\nu}(s) + (2\lambda)^{\frac{1}{\nu}} H_2^\nu(s) \ge t \right\}, \qquad t > 0,
	\end{equation}
	its distribution is related to that of $\mathcal{H}^\nu(t), t > 0$, by means of the formula
	\begin{equation}\label{1.14}
		\Prob\{\mathcal{L}^\nu(t) < x\} = \Prob\{\mathcal{H}^\nu(x) > t\}.
	\end{equation}

\subsection{Equation for the density of the inverse process}

	Let $\ell_\nu(x,t)$ be the probability density of the process of $\mathcal{L}^\nu(t)$, $t > 0$. The following result was stated in \cite{DovOrsToaldo} (see also \cite{OT2017}).
	
\begin{theorem}[\cite{DovOrsToaldo}]\label{densinv}	
	{The density $\ell_\nu(x,t)$ of the process $\mathcal{L}^\nu(t)$, $t > 0$, solves the time-fractional boundary-initial problem}
	\begin{equation}\label{2.20}
		\begin{cases}
			\left( \mathbb{D}_t^{2\nu} + 2\lambda \mathbb{D}_t^{\nu}\right) \ell_\nu(x,t) = -\frac{\partial}{\partial x} \ell_\nu(x,t), \quad x > 0,\,\, t > 0,\,\, 0 < \nu < \frac{1}{2}, \\
			\ell_\nu(x,0) = \delta(x), \\
			\ell_\nu(0,t) = \frac{t^{-2\nu}}{\Gamma(1-2\nu)} + 2\lambda \frac{t^{-\nu}}{\Gamma(1-\nu)},
		\end{cases}
	\end{equation}
	{and has $x$-Laplace transform, for $0 < \gamma < \lambda^2$,}
	\begin{equation}\label{2.21}
		\widetilde{\ell}_\nu(\gamma, t) = \frac{1}{2} \left[ \left( 1 + \frac{\lambda}{\sqrt{\lambda^2 - \gamma}} \right) E_{\nu,1} (r_1 t^\nu) + \left( 1 - \frac{\lambda}{\sqrt{\lambda^2 - \gamma}} \right) E_{\nu,1} (r_2 t^\nu) \right],
	\end{equation}
	{where}
	\begin{equation}\label{2.22}
		r_1 = -\lambda + \sqrt{\lambda^2 - \gamma}, \quad r_2 = -\lambda - \sqrt{\lambda^2 - \gamma}.
	\end{equation}
	{The fractional derivatives appearing in \eqref{2.20} are  in the Riemann-Liouville sense.}
\end{theorem}

	\begin{remark}
		The $x$-Laplace transform of the density $l_\nu(x,t)$ can be also given in terms of the multivariate generalized Mittag-Leffler function \eqref{multML} as follows:
		
		\begin{equation}\label{eq:xLaplace}
			\tilde{\ell}_{\nu}(\gamma, t) = E^{1,1}_{\nu, 1} (r_1 t^{\nu}, r_2 t^{\nu}) + 2\lambda t^{\nu} E^{1,1}_{\nu, \nu + 1}(r_1 t^{\nu}, r_2 t^{\nu}),
		\end{equation}
		where $r_1$ and $r_2$ are given by \eqref{2.22}.
		
		Indeed, the double Laplace transform of $l_{\nu}(t,x)$ is of the form
		
		\begin{equation}\label{eq:DDLT}
			\tilde{\tilde{\ell}}_{\nu}(\gamma, \mu) = \frac{\mu^{2\nu - 1} + 2\lambda \mu^{\nu - 1}}{\mu^{2\nu} + 2\lambda \mu^{\nu} + \gamma}\, ,
		\end{equation}
		which can be derived from the equation \eqref{DftG}, or, alternatively, by considering the $t$-Laplace transform of the density $\ell_{\nu}(t,x)$ given by   
		\eqref{t-lapl-inv} and then taking the $x$-Laplace transform.
		
		The expression \eqref{eq:DDLT} can be written as
		\begin{equation*}
			\frac{\mu^{2\nu - 1} + 2\lambda \mu^{\nu - 1}}{(\mu^{\nu} - r_1)(\mu^{\nu} - r_2)},
		\end{equation*}
		where $r_1$ and $r_2$ are defined in \eqref{2.22}. Applying the formula 
		\eqref{multMLfLapl} we obtain \eqref{eq:xLaplace}.
	\end{remark}

\begin{remark}
We presented here the equation for the density of  inverse process $\mathcal{L}^\nu(t)$, since it will be used in our study. The density of the process $\mathcal{H}^\nu(t)$ is also satisfies the fractional equation involving the sum of the R-L fractional derivatives,  w.r.t. the space variable (see, for details \cite{OT2017}, \cite{T}).  
\end{remark}

\subsection{Fractional Poisson processes  corresponding to time-change by the inverse to the sum of stable subordinators}

Consider the time-changed process $N^\nu(t)=N(\mathcal{L}^\nu(t))$, where $N$ is the Poisson process with the rate $\Lambda$ and $\mathcal{L}^\nu(t)$ is the inverse process defined in \eqref{1.13},  independent of $N$.

\begin{theorem}\label{fPequ}
	The probabilities ${p}^\nu_k(t)=\Prob\{N(\mathcal{L}^\nu(t))=k\}$ satisfy the following  equation
	\begin{equation}\label{(3.1)-(3.2)}
		\frac{d^{2\nu} p_k^\nu(t)}{dt^{2\nu}} + 2\lambda \frac{d^\nu p_k^\nu(t)}{dt^\nu} = -\Lambda\left(p_k^\nu(t) - p_{k-1}^\nu(t)\right), \quad k \ge 0,
	\end{equation}
	with the standard initial condition and the  fractional derivatives in the C-D sense.
\end{theorem}

\begin{proof}
	Equation \eqref{(3.1)-(3.2)} can be stated, in particular, in the similar way as the general result for the Poisson process time-changed by an inverse subordinators (see, e.g. \cite{BS2019}, \cite{BS2024}), as soon as we have the equations for their densities, which are provided in this case by \eqref{2.20}.
	
	For the probabilities $p_k^\nu(t)$ we have:
	\begin{align}
	p_k^\nu(t)&=\Prob\left\{N\left(\mathcal{L}^\nu(t)\right)=k\right\}
	=\int_{0}^{\infty}p_k(u)\ell_\nu(t,u)du,\nonumber
	\quad k=0,1,2,\dots
	\end{align}
	In view of equations \eqref{2.20} for the density $\ell_\nu(t,u)$ of the inverse subordinator $\mathcal{L}^\nu(t)$, applying the R-L derivatives, we obtain:
	\begin{align}
	\left( \mathbb{D}_t^{2\nu} + 2\lambda \mathbb{D}_t^{\nu}\right) p_k^\nu(t)
	&=\int_{0}^{\infty}p_k(u)\left( \mathbb{D}_t^{2\nu} + 2\lambda \mathbb{D}_t^{\nu}\right)\ell_\nu(t,u)du=- \int_{0}^{\infty}p_k(u)\frac{\partial}{\partial  u}\ell_\nu(t,u)du\nonumber\\
	&=\int_{0}^{\infty}\ell_\nu(t,u)\frac{\partial}{\partial  u}p_k(u)du-p_k(u)\ell_\nu(t,u)\big|_{u=0}^\infty\nonumber\\
	&=\int_{0}^{\infty}\ell_\nu(t,u)(-\Lambda[p_k(u)-p_{k-1}(u)])du+p_k(0)\ell_\nu(t,0)\nonumber\\
	&=-\Lambda\left[p^\nu_k(t)-p^\nu_{k-1}(t)\right]+p_k(0)\Big[\frac{t^{-2\nu}}{\Gamma(1-2\nu)} + 2\lambda \frac{t^{-\nu}}{\Gamma(1-\nu)}  \Big]. \label{11}
	\end{align}
	Using the relation \eqref{fDtrelation}
	    between the derivatives of C-D and R-L types, we have:
	\begin{equation}\label{12}
	\frac{d^{2\nu} p_k^\nu(t)}{dt^{2\nu}} + 2\lambda \frac{d^\nu p_k^\nu(t)}{dt^\nu} = \left( \mathbb{D}_t^{2\nu} + 2\lambda \mathbb{D}_t^{\nu}\right) p_k^\nu(t)-\Big[\frac{t^{-2\nu}}{\Gamma(1-2\nu)} + 2\lambda \frac{t^{-\nu}}{\Gamma(1-\nu)}  \Big]p_k^\nu(0),
	\end{equation}
	 note also that
	\begin{equation}\label{13}
	p_k^\nu(0)=\int_{0}^{\infty}p_k(u)\ell_\nu(0,u)du=\int_{0}^{\infty}p_k(u)\delta(u)du=p_k(0).
	\end{equation}
	From \eqref{11}, taking into account \eqref{12}-\eqref{13}, we  obtain \eqref{(3.1)-(3.2)}.
\end{proof}

To write the expressions for the probabilities ${p}_k^\nu(t)$ we will distiguish two cases. 

\medskip

For the case $\Lambda=\lambda^2$ in equation  \eqref{(3.1)-(3.2)}, 
 the expression for the probabilities ${p}_k^\nu(t)$  where calculated in \cite{BO2010} in terms of generalized Mittag-Leffler function \eqref{ML3} as presented in the next theorem.

\begin{theorem}[\cite{BO2010}]\label{fP2}
	The solution ${p}_k^\nu(t)$, for $k = 0, 1, \dots$ and $t \ge 0$, of \eqref{(3.1)-(3.2)}, with $\Lambda=\lambda^2$, is given by
	\begin{equation}\label{34}
		{p}_k^\nu(t) = \lambda^{2k} t^{2k\nu} E_{\nu, 2k\nu+1}^{2k+1}(-\lambda t^\nu) + \lambda^{2k+1} t^{(2k+1)\nu} E_{\nu, (2k+1)\nu+1}^{2k+2}(-\lambda t^\nu).
	\end{equation}
\end{theorem}

\medskip
The case  $\Lambda<\lambda^2$ is treated in the theorem below.

\begin{theorem}\label{fP3}
	The solution ${p}_k^\nu(t)$, for $k = 0, 1, \dots$ and $t \ge 0$, of \eqref{(3.1)-(3.2)}, with $\Lambda<\lambda^2$, is given by
	\begin{equation}\label{34.2}
		{p}_k^\nu(t) = \lambda^{2k} t^{2k\nu} E_{\nu, 2k\nu+1}^{k+1, k+1}(p_1 t^\nu, p_2 t^\nu) + \lambda^{2k+1} t^{(2k+1)\nu} E_{\nu, (2k+1)\nu+1}^{k+1, k+1}(p_1 t^\nu, p_2 t^\nu),
	\end{equation}
	where
	\begin{equation}\label{pp}
	p_1 = -\lambda + \sqrt{\lambda^2 - \Lambda}, \quad p_2 = -\lambda - \sqrt{\lambda^2 - \Lambda},
	\end{equation}
	and the multivariate Mittag-Leffler function is defined in \eqref{multML}.
\end{theorem}

	\begin{proof}
		Analogously to the proof of Theorem \ref{fP2} (see \cite{BO2010}), we perform the Laplace transform of the equation \eqref{(3.1)-(3.2)} and take into account the initial condition $p^{\nu}_k(0)$ for $k=0$, $p_k(0)=0$ for $k \geq 1$. We obtain the following expression for the Laplace transform of the solution:
		
		\begin{equation}\label{eq:fP3_Laplace}
			\mathcal{L}(p^{\nu}_{k}(t))(s) = \frac{\lambda^{2k} s^{2\nu - 1} + 2\lambda^{2k + 1}s^{\nu}}{(s^{2\nu} + 2\lambda s^{\nu} + \Lambda)^{k+1}} = \frac{\lambda^2k s^{2\nu - 1} + 2\lambda^{2k+1}s^{\nu}}{(s^{\nu} - p_1)^{k+1} (s^{\nu} - p_2)^{k+1}},   
		\end{equation}
with $p_1$ and $p_2$ defined in \eqref{pp}.		
				Then we invert \eqref{eq:fP3_Laplace} by using \eqref{multML} and come to the expression \eqref{34.2}.
	\end{proof}

\bigskip


We present in the next theorem the probability generating function of the process $N(\mathcal{L}^\nu(t)$ and its governing equation.

\begin{theorem}\label{thpgfN}
	The probability generating function $
	{G}_\nu(u, t) = \sum_{k=0}^{\infty} u^k {p}_k^\nu(t)$, $|u| \leq 1$, solves the following fractional differential equation
	\begin{equation}\label{pgfNequ}
		\frac{\partial^{2\nu} G_\nu(u, t)}{\partial t^{2\nu}} + 2\lambda \frac{\partial^\nu G_\nu(u, t)}{\partial t^\nu} = \Lambda(u - 1)G_\nu(u, t), 
	\end{equation}
	subject to the initial condition $G(u, 0) = 1$, and is given as
	\begin{equation}\label{pgfN}
		G_\nu(u, t) = \tilde \ell_\nu (\Lambda(u - 1), t), \,\,u < 1,
	\end{equation}
	where the expression for $\tilde \ell_\nu$ is given by \eqref{2.21}.
	In particular, for the case $\Lambda=\lambda^2$, \eqref{pgfN} becomes
	\begin{equation}\label{pgfN2}
		{G}_\nu(u, t) = \frac{\sqrt{u} + 1}{2\sqrt{u}} E_{\nu, 1}(-\lambda(1 - \sqrt{u})t^\nu) + \frac{\sqrt{u} - 1}{2\sqrt{u}} E_{\nu, 1}(-\lambda(1 + \sqrt{u})t^\nu).
	\end{equation}
\end{theorem}

\begin{proof}
The result follows from the general Theorem \ref{Th5}, by using formulas \eqref{Gpsif1} and \eqref{DftG}, where the expression for the Laplace transform of the inverse subordinator is provided in our case by Theorem \ref{densinv}, formula \eqref{2.21}, from which for the particular case $\Lambda=\lambda^2$ the  expression \eqref{pgfN2}	 can be calculated. 
Note that for the  case  $\Lambda=\lambda^2$ equation \eqref{pgfNequ} and the expression \eqref{pgfN2} were also derived in \cite{BO2010}, within a different approach. 
\end{proof}

\begin{remark}
Note that to represent the probability generating function $G_\nu(u, t)$ we can also use other expression for $\tilde \ell_\nu$ which is  given by \eqref{eq:xLaplace}. 
\end{remark}

The case of nonhomogeneous Poisson process can be treated similarly  to the paper \cite{BS2019}. 

Let  $\mathsf{N}(t), t\geq0$, be a non-homogeneous Poisson process with intensity function $\lambda(t):[0,\infty)\to[0,\infty)$. Denote its marginal distribution $\mathsf p_n(u)$, and 
$\Lambda(t)=\Lambda(0,t)$. Consider the time-changed process $\mathsf{N}^\nu(t)=\mathsf{N}(\mathcal{L}^\nu(t))$, where $\mathcal{L}^\nu(t)$ is the inverse process defined in \eqref{1.13},  indepedent of $\mathsf{N}$. Then we have the marginal distributions
$$
   \mathsf p_n^\nu(t)=\Prob\left\{\mathsf{N}\left(\mathcal{L}^\nu(t)\right)=n\right\}
    =\int_{0}^{\infty}\mathsf p_n(u)\ell_\nu(t,u)du
    =\int_{0}^{\infty}\frac{e^{-\Lambda(u)}\Lambda(u)^n}{n!}\ell_\nu(t,u)du, \, n=0,1,2,\dots,
$$
where $\ell_\nu(t,u)$ is the density of the process $\mathcal{L}^\nu(t)$. The next theorem  presents the governing equations for  $\mathsf p_n^\nu(t)$, the proof   follows by the same arguments as those for Theorem 2 in \cite{BS2019}. 
\begin{theorem}
     The marginal distributions $\mathsf p_n^\nu(t)=\Prob\left\{\mathsf{N}\left(\mathcal{L}^\nu(t)\right)=n\right\}$, $n=0,1,\dots,$ satisfy the differential-integral equations
    \begin{equation}\label{fDp}
    \frac{d^{2\nu} \mathsf p_k^\nu(t)}{dt^{2\nu}} + 2\lambda \frac{d^\nu \mathsf p_k^\nu(t)}{dt^\nu}=\int_{0}^{\infty}\lambda(u)\left[-\mathsf p_k(u)+ \mathsf p_{k-1}(u)\right]\ell_\nu(t,u)du,
    \end{equation}
    with the usual initial condition
    $
    \mathsf p^\nu_k(0)=    
    1$, for $k=0$, $
    \mathsf p^\nu_k(0)=
    0$, for $k\geq1$,
    $\mathsf p_{-1}^\nu(0)=0$, where the derivatives are in the  C-D sense,  $\ell_\nu(t,u)$ is the density of the inverse subordinator $\mathcal{L}^\nu(t)$.
\end{theorem}

\subsection{Generalized fractional counting processes  corresponding to time-change by the inverse to the sum of stable subordinators}

Consider the time-changed process $M^\nu(t)=M(\mathcal{L}^\nu(t))$, where $M$ is the generalized counting process with parameters $\lambda_1, \ldots, \lambda_k$, and $\mathcal{L}^\nu(t)$ is the inverse process defined in \eqref{1.13},   independent of $M$.

We have the following result for the probabilities ${\tilde p}_n^\nu(t)=\Prob\{M(\mathcal{L}^\nu(t))=n\}$.

\begin{theorem}\label{fM2}
The probabilities ${\tilde p}_n^\nu(t)$
satisfy the following  equation
\begin{equation}\label{fM2eq}
	\frac{d^{2\nu} \tilde p_n^\nu}{dt^{2\nu}} + 2\lambda \frac{d^\nu \tilde p_n^\nu}{dt^\nu} = - \Lambda \tilde p_n(t) + \sum_{j=1}^{\min\{n,k\}}\lambda_j \tilde p_{n-j}(t), \,\, n\geq0,
\end{equation}
with the standard initial condition and fractional derivatives in C-D sense.

	The solution ${\tilde p}_n^\nu(t)$, for $n = 0, 1, \dots$ and $t \ge 0$, of \eqref{fM2eq} is given as follows:
	\begin{description}
	\item $\quad$if $\lambda=\Lambda^{1/2}$,
	\begin{align}\label{pGFCP}
		\tilde p_n^\nu(t) &= \sum_{r=0}^n \sum_{\substack{\alpha_1+\alpha_2+\dots+\alpha_k=r \\ \alpha_1+2\alpha_2+\dots+k\alpha_k=n}} \binom{r}{\alpha_1, \alpha_2, \dots, \alpha_k} \lambda_1^{\alpha_1} \lambda_2^{\alpha_2} \dots \lambda_k^{\alpha_k} \nonumber \\
				&\times  \left(t^{2r\nu} E_{\nu, 2r\nu+1}^{2r+1}(-\Lambda^{1/2} t^\nu) + \Lambda^{1/2} t^{(2r+1)\nu} E_{\nu, (2r+1)\nu+1}^{2r+2}(-\Lambda^{1/2} t^\nu)\right),
	\end{align}
	where the generalized Mittag-Leffler function is defined in \eqref{ML3};
		\item $\quad$if $\lambda>\Lambda^{1/2}$,
		\begin{align}\label{pGFCP2}
		\tilde p_n^\nu(t) &= \sum_{r=0}^n \sum_{\substack{\alpha_1+\alpha_2+\dots+\alpha_k=r \\ \alpha_1+2\alpha_2+\dots+k\alpha_k=n}} \binom{r}{\alpha_1, \alpha_2, \dots, \alpha_k} \lambda_1^{\alpha_1} \lambda_2^{\alpha_2} \dots \lambda_k^{\alpha_k} \nonumber \\
				&\times  \left(\lambda^{2k} t^{2k\nu} E_{\nu, 2k\nu+1}^{k+1, k+1}(p_1 t^\nu, p_2 t^\nu) + \lambda^{2k+1} t^{(2k+1)\nu} E_{\nu, (2k+1)\nu+1}^{k+1, k+1}(p_1 t^\nu, p_2 t^\nu)\right),
	\end{align}		
		where
		$p_1$ and  $p_2$ are given in \eqref{pp}
		and the multivariate Mittag-Leffler function is defined in \eqref{multML}.
	\end{description}
	
\end{theorem}

\begin{proof}
Equation \eqref{pGFCP} for the probabilities ${\tilde p}_n^\nu(t)=\Prob\{M(\mathcal{L}^\nu(t))=n\}$  is derived by following the same lines as in the proof of Theorem \ref{fPequ}, and using the governing equation for the generalized  counting process \eqref{pGCP}.
	To prove \eqref{pGFCP} and \eqref{pGFCP2} we use the representation $M^\nu(t) = \sum_{i=1}^{N_\Lambda^\nu(t)} X_i$, where $N_\Lambda^\nu(t)=N_\Lambda(\mathcal{L}^\nu(t))$,  $\mathcal{L}^\nu(t)$ is the inverse process defined in \eqref{1.13}, $X_1, X_2, \dots, X_r$ are independent and identically distributed random variables described in Theorem \ref{thmDic}. By conditioning arguments we have
	\begin{align*}
		\tilde p_j^\nu(t) = \mathbb{P}\{M^\nu(t) = j\} = \sum_{r=0}^j \mathbb{P}\{X_1 + X_2 + \dots + X_r = j\} \mathbb{P}\{N_\Lambda^\nu(t) = r\},
	\end{align*}
where
	\begin{align*}
		\mathbb{P}\{X_1 + X_2 + \dots + X_r = j\} &= \sum_{\substack{\alpha_1+\alpha_2+\dots+\alpha_k=r \\ \alpha_1+2\alpha_2+\dots+k\alpha_k=j}} \binom{r}{\alpha_1, \alpha_2, \dots, \alpha_k} 
		\left(\frac{\lambda_1}{\Lambda}\right)^{\alpha_1} \left(\frac{\lambda_2}{\Lambda}\right)^{\alpha_2} \dots \left(\frac{\lambda_k}{\Lambda}\right)^{\alpha_k}.
	\end{align*}
	Then we substitute the  expressions for the distribution $\mathbb{P}\{N_\Lambda^\nu(t) = r\}$ from \eqref{34} or from \eqref{34.2} for  the cases $\lambda=\Lambda^{1/2}$ and $\lambda>\Lambda^{1/2}$ and obtain \eqref{pGFCP} and \eqref{pGFCP2} correspondingly.    
\end{proof}

Similarly to the case of time-changed Poisson process, the probability generating function of the process $M(\mathcal{L}^\nu(t))$ and its governing equation can be obtained by applying again  the general results from Theorem \ref{Th5},  formulas \eqref{Gpsif1} and \eqref{DftG}, and appealing to  the expression for the Laplace transform of the inverse subordinator $\mathcal{L}^\nu(t)$ known from  Theorem \ref{densinv}, formula \eqref{2.21}. 

\begin{theorem}\label{thpgfM}
	The probability generating function $
	\tilde {G}_\nu(u, t) = \sum_{k=0}^{\infty} u^k \tilde{p}_k^\nu(t)$, $|u| \leq 1$, solves the following fractional differential equation
	\begin{equation}\label{pgfMequ}
		\frac{\partial^{2\nu} \tilde G_\nu(u, t)}{\partial t^{2\nu}} + 2\lambda \frac{\partial^\nu \tilde G_\nu(u, t)}{\partial t^\nu} = -\sum_{j=1}^{k}\lambda_j(1-u^j)\tilde G_\nu(u, t), \,\,|u| < 1,
	\end{equation}
	subject to the initial condition $\tilde {G}_\nu(u, 0) = 1$, and is given as
	\begin{equation}\label{pgfM}
		\tilde G_\nu(u, t) = \tilde \ell_\nu \Big(\sum_{j=1}^{k}\lambda_j(1-u^j), t\Big), \,\,|u| < 1,
	\end{equation}
	where the expression for $\tilde \ell_\nu$ is given by \eqref{2.21}  $($or by \eqref{eq:xLaplace}$)$.
	\end{theorem}

\begin{remark}
More general time-changed processes can be considered of the form   ${M}^{\psi,\nu}(t)=M(H^{\psi}(\mathcal{L}^\nu(t))$, that is, the generalized counting process with double time-change by an independent subordinator $H^{\psi}$ and an inverse subordinator $\mathcal{L}^\nu(t)$, which are independent of $M$. Then the probabilities ${p}_n^{\psi,\nu}(t)=\Prob\left\{{M}^{\psi,\nu}(t)=n\right\}$ and the probability generating function of ${M}^{\psi,\nu}$ will be governed by  the equations of the form \eqref{fM2eq} and \eqref{pgfMequ}, where the left hand sides are retained, that is, with that telegraph-type operator in time variable, but the right hand sides will be given by those in equations \eqref{Dtfpn} and \eqref{Gpsif1} in Theorem \ref{Th5}.
\end{remark}

\begin{remark}
Consider  the time-changed process $\mathsf{M}(\mathcal{L}^\nu(t))$, where $\mathsf{M}$ is a nonhomogeneous generalized counting process with intensity functions $\lambda_j(t):[0,\infty)\to[0,\infty)$, defined in \cite{KK2022Arxiv},  and the inverse process $\mathcal{L}^\nu(t)$ is  defined in \eqref{1.13}. 
Then 
     the marginal distributions ${\tilde{\mathsf p}}_n^\nu(t)=\Prob\left\{\mathsf{M}\left(\mathcal{L}^\nu(t)\right)=n\right\}$, $n=0,1,\dots,$ satisfy the differential-integral equations
    \begin{equation*}
    \frac{d^{2\nu} {\tilde{\mathsf p}}_n^\nu(t)}{dt^{2\nu}} + 2\lambda \frac{d^\nu {\tilde{\mathsf p}}_n^\nu}{dt^\nu}=\int_{0}^{\infty}\Bigl(-\sum_{j=1}^{k}\lambda_j(u){\tilde{\mathsf p}}^\nu_n(u)+ \sum_{j=1}^{\min\{n,k\}}\lambda_j(u){\tilde{\mathsf p}}^\nu_{n-j}(u)\Bigr)\ell_\nu(t,u)du,
    \end{equation*}
    with the usual initial condition,
        where the derivatives are in the  C-D sense,  $\ell_\nu(t,u)$ is the density of the inverse subordinator $\mathcal{L}^\nu(t)$.
        
        Note that for particular case where $\lambda_j(t)=\lambda(t)$ and $\lambda_j(t)=(1-\rho)\rho^{j-1}\lambda(t)/(1-\rho^k)$, $j=1,\ldots, k$ we obtain the time changed non-homogeneous Poisson process of order $k$ and the non-homogeneous P\'olya-Aeppli process of order $k$ respectively, which generalize time-changed models of these processes considered in \cite{KLS}.

\end{remark}

\subsection{Further generalizations}

The previous results can be generalized by considering the linear combinations of stable subordinators of the form 
\begin{equation}\label{genHnu}
		\mathcal{H}^\nu(t)=\mathcal{H}^{\nu_1,\ldots,\nu_N}(t) = \sum_{j=1}^N (\mu_j)^{\frac{1}{\nu_j}}H_j^{\nu_j}(t), \quad t > 0,\,\,\mu_j>0,\,\,\,  \nu_j \in (0,1),
	\end{equation}
	where  $H_j^{\nu_j}(t)$  are independent stable subordinators of orders ${\nu_j}$, and  the corresponding  inverse  process 
	\begin{equation}\label{genLnu}
	\mathcal{L}^\nu(t)=\mathcal{L}^{\nu_1,\ldots,\nu_N}(t)  = \inf \left\{ s > 0 : \mathcal{H}^{\nu_1,\ldots,\nu_N}(t)  \ge t \right\}, \qquad t > 0.
	\end{equation}

	 Let $\ell_\nu(x,t)=\ell_{\nu_1,\ldots,\nu_m}(x,t)$ denote now the probability density of the process \eqref{genLnu}. The following result was stated in \cite{OT2017}.

\begin{theorem}[\cite{OT2017}]\label{geninvequ}	
	{The density $\ell_\nu(x,t)$ of the process $\mathcal{L}^\nu(t)$, $t > 0$, solves the time-fractional boundary-initial problem}
	\begin{equation}\label{genLdensity}
		\begin{cases}
			\sum_{j=1}^N \,\mu_j\,\mathbb{D}_t^{\nu_j} \, \ell_\nu(x,t) = -\frac{\partial}{\partial x} \ell_\nu(x,t), \quad x > 0,\,\, t > 0,\,\, \mu_j>0,\,\,\nu_j \in (0,1), \\
			\ell_\nu(x,0) = \delta(x), \\
			\ell_\nu(0,t) = \sum_{j=1}^N \mu_j\,\frac{t^{-\nu_j}}{\Gamma(1-\nu_j)} ,
		\end{cases}
	\end{equation}
	with the fractional derivatives in the R-L sense.
\end{theorem}

Consider the time-changed process ${\mathcal N}^\nu(t)={\mathcal N}(\mathcal{L}^\nu(t))$, where $\mathcal N$ is the Poisson process with the rate $\Lambda$ and $\mathcal{L}^\nu(t)$ is now the inverse process defined in \eqref{genLnu}, independent of $\mathcal N$.

Analogously to Theorem \ref{fPequ}, and by using Theorem \ref{geninvequ}, the following result can be stated.

\begin{theorem}\label{genfPequ} The probabilities ${p}^\nu_k(t)=\Prob\{\mathcal N(\mathcal{L}^\nu(t))=k\}$
satisfy the following  equation
\begin{equation}\label{gen(3.1)-(3.2)}
	\sum_{j=1}^N \,\mu_j \frac{d^{\nu_j} }{dt^{\nu_j}}\, p_k^\nu(t) = -\Lambda\left(p_k^\nu(t) - p_{k-1}^\nu(t)\right), \quad k \ge 0,
\end{equation}
with the standard initial condition and the  fractional derivatives in the C-D sense.
\end{theorem}

Under  particular conditions on the parameters $\nu_j$, $\mu_j$, $j=1,\ldots, N$, and $\Lambda$, the probabilities $p_k^\nu(t)$ can be calculated as shown in  the theorem below.

\begin{theorem}
	Let $\mathcal N_\Lambda(\mathcal{L}^\mu(t))$ be the time-changed Poisson process with the rate $\Lambda$, where $\mathcal{L}^\nu(t)$ is the inverse process defined in 
	\eqref{genLnu} with parameters $\mu_j\ge0$, $j=1,\dots,N$ such that the following holds:
	\begin{equation} \label{eq:T12}
		\sum_{j=1}^N \mu_j x^j + \Lambda = \prod_{j=1}^M (x-\eta_j)^{m_j}, \quad \eta_j \in \mathbb{R}, \quad \sum_{j=1}^M m_j=N,
	\end{equation}
	and $\nu_j$ are of the form $\nu_j = j\nu$, $j=1,\dots,N$, for some $\nu \leq \frac{1}{N}$.
	
	Then $p_k^\nu(t)$ which solve equation 
	\eqref{gen(3.1)-(3.2)} can be represented as follows:
	\begin{equation} \label{eq:T12_2}
		p_k^\nu(t) = \Lambda^k \sum_{j=1}^N \mu_j t^{\delta_j - 1} E_{\nu, \delta_j}^{\gamma_1, \dots, \gamma_M} (\eta_1 t^\nu, \dots, \eta_M t^\nu),
	\end{equation}
	where $\gamma_j = m_j(k+1)$, $\delta_j = \nu(N(k+1)-j)+1$, $E_{\nu, \delta_j}^{\gamma_1, \dots, \gamma_M}$ is the multivariate Mittag-Leffler function defined by \eqref{multML}.
			\end{theorem}

\begin{proof}
	Taking the Laplace transform of the equation 
	\eqref{gen(3.1)-(3.2)} we obtain (denoting the Laplace transform as $\hat p$):
	\begin{equation*}
		\sum_{j=1}^N \mu_j s^{\nu j} \hat{p}_k^\nu(s) - \sum_{j=1}^N \mu_j s^{\nu j - 1} p_k^\nu(0) = -\Lambda(\hat{p}_k^\nu(s) - \hat{p}_{k-1}^\nu(s)).
	\end{equation*}
	In view of the initial condition, we have $p_k^\nu(0)=1$ for $k=0$ and $p_k^\nu(0)=0$ for $k \geq 1$, therefore,
	\begin{equation*}
		\hat{p}_0^\nu(s) = \frac{\sum_{j=1}^N \mu_j s^{\nu j - 1}}{\sum_{j=1}^N \mu_j s^{\nu j} + \Lambda}
	\end{equation*}
	and for $k \geq 1$
	\begin{equation*}
		\hat{p}_k^\nu(s) = \frac{\Lambda}{\sum_{j=1}^N \mu_j s^{\nu j} + \Lambda} \hat{p}_{k-1}(s) = \frac{\Lambda^k \sum_{j=1}^N \mu_j s^{\nu j - 1}}{\left(\sum_{j=1}^N \mu_j s^{\nu j} + \Lambda\right)^{k+1}}\, .
	\end{equation*}
	Now we apply \eqref{eq:T12} and obtain
	\begin{equation} \label{eq:T12_3}
		\hat{p}_k^\nu(s) = \frac{\Lambda^k \sum_{j=1}^N \mu_j s^{\nu j - 1}}{\prod_{j=1}^M (s^\nu - \eta_j)^{m_j(k+1)}}\, .
	\end{equation}
	
	To invert \eqref{eq:T12_3} we use the formula \eqref{multMLfLapl} and come to the expression \eqref{eq:T12_2}.
\end{proof}

\begin{remark}
One immediate choice for the collection of $\mu_j$, $j=1,\ldots, N$, and $\Lambda$ to satisfy \eqref{eq:T12} comes in relation with probability generating function of a sum $X=\sum_{i=1}^{N} X_i$ of independent Bernoulli random variables $X_i$ with $\Prob\{X_i=1\}=b_i$, $\Prob\{X_i=0\}=1-b_i$. Let $p_k=\Prob\{X=k\}$. Then we have $\sum_{i=0}^{N} p_k x^k=\prod_{i=1}^{N} b_i\prod_{i=1}^{N}\big(x+\frac{1-b_i}{b_i}\big)$, or $\big(\prod_{i=1}^{N} b_i\big)^{-1}\sum_{i=0}^{N} p_k x^k=\prod_{i=1}^{N}\big(x+\frac{1-b_i}{b_i}\big)$, which gives options to choose values $\mu_j$, $j=1,\ldots, N$, and $\Lambda$, and corresponding $\eta_i$, such that \eqref{eq:T12} holds. Note that $b_i$ may be all distinct, or all equal, or just some of them could coincides, thus, giving a variety of representations.

With the particular choice $\mu_j=\binom{N}{j}\lambda^{N-j}$, $\Lambda=\lambda^n$, \eqref{eq:T12} becomes
			$\sum_{j=1}^N \mu_j x^j + \Lambda =  (x+\lambda)^N$.
			  
To inverst \eqref{eq:T12_3} in this case, we can use the formula
\begin{equation}
		\mathscr{L} \left\{ t^{\gamma-1} E_{\beta, \gamma}^\delta (\omega t^\beta); s \right\} = \frac{s^{\beta\delta - \gamma}}{(s^\beta - \omega)^\delta},
	\end{equation}
	(where $Re(\beta) > 0$, $Re(\gamma) > 0$, $Re(\delta) > 0$ and $s > |\omega|^{\frac{1}{Re(\beta)}}$).			
			 
			 Therefore, in this case
	 $p_k^\nu(t)$  can be represented as follows:
\begin{center}	 
	$
		p_k^\nu(t) = \sum_{j=1}^N \binom{N}{j}  (\lambda t)^{\delta_j - 1} E_{\nu, \delta_j}^{N(k+1)} (-\lambda t^\nu),
		$
		\end{center}
where $\delta_j=\nu(N(k+1)-j)+1$, $E_{\nu, \delta}^{\gamma}$ is the Mittag-Leffler function defined by \eqref{ML3}.
\end{remark}

\medskip

We can next consider the time-changed process $M^\nu(t)=M(\mathcal{L}^\nu(t))$, where $M$ is the generalized counting process with parameters $\lambda_1, \ldots, \lambda_k$, and $\mathcal{L}^\nu(t)$ is the inverse process defined in 
	\eqref{genLnu} with parameters $\mu_j$, $j=1,\dots,N$,   independent of $M$.

Then we can state the following result for the probabilities ${\tilde p}_n^\nu(t)=\Prob\{M(\mathcal{L}^\nu(t))=n\}$.

\begin{theorem}\label{genfM2}
The probabilities ${\tilde p}_n^\nu(t)$
satisfy the following  equation
\begin{equation}\label{genfM2eq}
	\sum_{j=1}^N \,\mu_j \frac{d^{\nu_j} }{dt^{\nu_j}}\, \tilde p_n^\nu = - \Lambda \tilde p_n(t) + \sum_{j=1}^{\min\{n,k\}}\lambda_j \tilde p_{n-j}(t), \,\, n\geq0,
\end{equation}
with the standard initial condition and the  fractional derivatives in the C-D sense.

If  \eqref{eq:T12} holds, then 
	the probabilities ${\tilde p}_n^\nu(t)$,  $n = 0, 1, \dots$, $t \ge 0$, can be given as follows:
		\begin{align*}\label{pGFCP}
		\tilde p_n^\nu(t) &= \sum_{r=0}^n \sum_{\substack{\alpha_1+\alpha_2+\dots+\alpha_k=r \\ \alpha_1+2\alpha_2+\dots+k\alpha_k=n}} \binom{r}{\alpha_1, \alpha_2, \dots, \alpha_k} \lambda_1^{\alpha_1} \lambda_2^{\alpha_2} \dots \lambda_k^{\alpha_k} \ p_r^\nu(t),
	\end{align*}
	where $p_k^\nu(t)$ are defined in \eqref{eq:T12_2}.
\end{theorem}


\section{Sum of tempered stable subordinators and its inverse \\and corresponding time-changed counting processes}

Consider the tempered stable subordinator $H^{\alpha,\rho}(t)$, with the  Bern\v{s}tein function
    \begin{equation}\label{exf}
    f(x)=f^{\alpha,\rho}(x)=\left(x+\rho\right)^{\alpha}-\rho^{\alpha}, \quad \alpha\in (0,1), \rho>0.
    \end{equation}
    The corresponding L\'evy measure is given by the formula:
    \begin{equation*}
    \overline{\nu}(ds)=\frac{1}{\Gamma(1-\alpha)}\alpha e^{-\rho s}s^{-\alpha-1}ds,
    \end{equation*}
    and its tail is
    \begin{equation*}
    \nu(s)=\frac{1}{\Gamma(1-\alpha)}\alpha \rho^{\alpha}\Gamma(-\alpha,s),
    \end{equation*}
    where $\Gamma(-\alpha,s)=\int_{s}^{\infty}e^{-z}z^{-\alpha-1}dz$ is the incomplete Gamma function.

    The generalized C-D convolution-type derivative \eqref{fDt} for $f(x)$ given by \eqref{exf} becomes:
    \begin{equation}\label{exfDt}
    {\mathcal D}^{\alpha,\rho}_t u(t)= \frac{\alpha \rho^{\alpha}}{\Gamma(1-\alpha)}\int_{0}^{t}\frac{\partial }{\partial t }u(t-s)\Gamma(-\alpha,s)ds,
    \end{equation}
and the corresponding generalized  R-L fractional derivative is given by the formula:
\begin{equation}\label{exfDDt}
\mathbb{D}^{\alpha,\rho}_t u(t)=\frac{\alpha \rho^{\alpha}}{\Gamma(1-\alpha)}\frac{d}{dt}\int_{0}^{t}u(t-s)\Gamma(-\alpha,s)ds
\end{equation}
(see, \cite{T}).

Consider now the sum
	\begin{equation}\label{temp2.1}
		\mathcal{H}^{\alpha,\rho}(t) = H_1^{2\alpha, \rho}(t) +  H_2^{\alpha,\rho}(t), \quad t > 0,\,\,\, 0 < \alpha \le \frac{1}{2},
	\end{equation}
	where  $H_1^{2\alpha, \rho}$, $H_2^{\alpha,\rho}$  are independent tempered stable subordinators with parameters ${2\alpha, \rho}$ and $\alpha, \rho$ correspondingly, $\rho > 0$. 
	
	Define the corresponding inverse process  $\mathcal{L}^{\alpha,\rho}(t), t > 0$, to the process $\mathcal{H}^{\alpha,\rho}(t), t > 0$:
	\begin{equation}\label{temp1.13}
		\mathcal{L}^{\alpha,\rho}(t) = \inf \left\{ s > 0 : H_1^{2\alpha,\rho}(s) +  H_2^{\alpha,\rho}(s) \ge t \right\}, \quad t > 0.
	\end{equation}
	
\subsection{Equation for the density of the inverse process}

Let $l_{\alpha,\rho}(x,t)$ be the probability density of the process of $\mathcal{L}^{\alpha,\rho}(t)$, $t > 0$. We can state the following result.
\begin{theorem}	
	{The density $l_{\alpha,\rho}(x,t)$  solves the time-fractional boundary-initial problem}
	\begin{equation}\label{temp3}
		\begin{cases}
			(\mathbb{D}^{2\alpha,\rho} + \mathbb{D}^{\alpha,\rho}) l_{\alpha, \rho}(x,t) = -\frac{\partial}{\partial x} l_{\alpha, \rho}(x,t), \quad x > 0,\,\, t > 0,\,\, 0 < \alpha < \frac{1}{2},  \\
			l_{\alpha, \rho}(x,0) = \delta(x), \\
			l_{\alpha, \rho}(0,t) = \frac{2\alpha \rho^{2\alpha}}{\Gamma(1-2\alpha)} \Gamma(-2\alpha, t) + \frac{\alpha \rho^\alpha}{\Gamma(1-\alpha)} \Gamma(-\alpha, t),
		\end{cases}
			\end{equation}
with the generalized R-L derivatives defined in \eqref{exfDDt}.
\end{theorem}	
	\begin{proof} Firsly, we find the double Laplace transform of the solution  to the problem \eqref{temp3}.
		Taking the $t$-Laplace transform of the equation \eqref{temp3} we have (we  will write below simply ${l}(x,s)$ omitting the subscript $_{\alpha,\rho}$):
		\begin{align*}
			f^{2\alpha, \rho}(s) \tilde{l}(x,s) + f^{\alpha, \rho}(s) \tilde{l}(x,s) = -\frac{\partial}{\partial x} \tilde{l}(x,s).
		\end{align*}
Then, with $x$-Laplace transform we obtain
		\begin{align*}
			(f^{2\alpha, \rho}(s) + f^{\alpha, \rho}(s)) \tilde{\tilde{l}}(\gamma, s) = \tilde{l}(0,s) - \gamma \tilde{\tilde{l}}(\gamma,s)
		\end{align*}
and, taking into account the boundary condition, we can calculate:
		\begin{align*}
			\tilde{l}(0,s) &= \int_0^{+\infty} e^{-st} l(0,t) dt = \int_0^{+\infty} e^{-st} (\nu_1(t) + \nu_2(t)) dt = \frac{f_1(s) + f_2(s)}{s},			
		\end{align*}
where we have denoted
$$\nu_1(t)=\frac{2\alpha \rho^{2\alpha}}{\Gamma(1-2\alpha)} \Gamma(-2\alpha, t), \,\, \nu_1(t)= \frac{\alpha \rho^\alpha}{\Gamma(1-\alpha)} \Gamma(-\alpha, t)$$	
and
$$f_1(s) = f^{2\alpha, \rho}(s) = (s+\rho)^{2\alpha} - \rho^{2\alpha}, \,\,
			f_2(s) = f^{\alpha, \rho}(s) = (s+p)^\alpha - \rho^\alpha.$$
Thus,				
		\begin{equation}\label{temp4}
			\tilde{\tilde{l}}(\gamma, s) = \frac{f_1(s) + f_2(s)}{s(f_1(s) + f_2(s) + \gamma)}
		\end{equation}
		
				On the other hand, let $l(t,x)$ be the density of the inverse process $\mathcal{L}^{\alpha, \rho}(t)$, then we can write
		\begin{equation}\label{(5)}
			l(t,x) = \frac{P(\mathcal{L}^{\alpha, \rho}(t) \in dx)}{dx} = -\frac{\partial}{\partial x} P \{ \mathcal{H}^{\alpha, \rho}(x) < t \} = -\frac{\partial}{\partial x} \int_0^t h(u,x) du,
		\end{equation}
where $h(t,x)$ is the probability density of the 	process	$\mathcal{H}^{\alpha, \rho}(t)=H_1^{2\alpha, \rho}(t) +  H_2^{\alpha,\rho}(t),$ which has the Laplace exponent $f_1(s) + f_2(s)$.

		In view of \eqref{(5)}, the double Laplace transform of $l(t,x)$ can be calculated as follows:
		\begin{align}
			\tilde{\tilde{l}}(\gamma, s) &= \int_0^{+\infty} e^{-\gamma x} \int_0^\infty e^{-st} \left[ -\frac{\partial}{\partial x} \int_0^t h(u,x) du \right] dt dx \nonumber\\
			&= - \int_0^\infty e^{-\gamma x} \frac{\partial}{\partial x} \int_0^\infty e^{-st} \int_0^t h(u,x) du dt dx = \nonumber\\
			&= -\frac{1}{s} \int_0^\infty e^{-\gamma x} \frac{\partial}{\partial x} \tilde{h}(s,x) dx = -\frac{1}{s} \int_0^\infty e^{-\gamma x} \frac{\partial}{\partial x} e^{-x(f_1(s) + f_2(s))} dx = \nonumber\\
			&= \frac{f_1(s) + f_2(s)}{s} \int_0^\infty e^{-\gamma x} e^{-x(f_1(s) + f_2(s))} dx = \frac{f_1(s) + f_2(s)}{s(f_1(s) + f_2(s) + \gamma)}.\label{LL}
		\end{align}
		
		Alternatively, we know  the expression \eqref{t-lapl-inv} for the  Laplace transform of the density of the inverse subordinator with respect to $t$ from which we can write the $t$-Laplace for the density  under consideration, and then applying $x$-transform we  get again \eqref{LL}.
		
		Therefore, the double Laplace transform of the density of the inverse subordinator $\mathcal{L}^{\alpha, p}(t)$ coincides with that of the solution of \eqref{temp3} given by \eqref{temp4}.\end{proof}

\subsection{Generalized counting processes  corresponding to time-change by the inverse to the sum of tempered stable subordinators}
 
Firstly, consider the time-changed process $N^{\alpha,\rho}(t)=N(\mathcal{L}^{\alpha,\rho}(t))$, where $N$ is the Poisson process with the rate $\lambda$ and $\mathcal{L}^{\alpha,\rho}(t)$ is the inverse process defined in \eqref{temp1.13}, independent of $N$. 

\begin{theorem}\label{fP2equ}
	The probabilities ${p}^{\alpha,\rho}_n(t)=\Prob\{N(\mathcal{L}^{\alpha,\rho}(t))=n\}$ satisfy the following  equation
	\begin{equation}\label{fPtemp}
		\left({\mathcal D}^{2\alpha,\rho}_t + {\mathcal D}^{\alpha,\rho}_t\right) p_n^{\alpha,\rho}(t) = -\lambda(p_n^{\alpha,\rho}(t) - p_{n-1}^{\alpha,\rho}(t)), \quad n \ge 0,
	\end{equation}
	with the standard initial condition, and  the probability generating function $
	{G}_{\alpha,\rho}(u, t)$, $|u| \leq 1$, solves the equation
	\begin{equation*}
		\left({\mathcal D}^{2\alpha,\rho}_t + {\mathcal D}^{\alpha,\rho}_t\right) G_{\alpha,\rho}(u, t) = \lambda(u - 1)G_{\alpha,\rho}(u, t), 
	\end{equation*}
	with $G_{\alpha,\rho}(u, 0) = 1$; the generalized fractional derivatives  in the C-D sense are defined in \eqref{exfDt}.
\end{theorem}
\begin{proof}
	Equation for probabilities \eqref{fPtemp} is derived by following the same lines as in the proof of Theorem \ref{fPequ},  using the  equation for the density of the inverse process \eqref{temp3}.
\end{proof}

Consider now the time-changed process $M^{\alpha,\rho}(t)=M(\mathcal{L}^{\alpha,\rho}(t))$, where $M$ is the generalized counting process and $\mathcal{L}^{\alpha,\rho}(t)$ is the inverse process defined in \eqref{temp1.13}, independent of $M$.

We have the following result for the probabilities ${\tilde p}_n^{\alpha,\rho}(t)=\Prob\{M(\mathcal{L}^{\alpha,\rho}(t))=n\}$.

\begin{theorem}
The probabilities ${\tilde p}_n^{\alpha,\rho}(t)$
satisfy the following  equation
\begin{equation}\label{fM2eqt}
	\left({\mathcal D}^{2\alpha,\rho}_t + {\mathcal D}^{\alpha,\rho}_t\right) \tilde p_k^{\alpha,\rho}(t) = - \Lambda \tilde p_n^{\alpha,\rho}(t) + \sum_{j=1}^{\min\{n,k\}}\lambda_j \tilde p_{n-j}^{\alpha,\rho}(t), \, n\geq0, 
\end{equation}
with the standard initial condition,
and  the corresponding probability generating function $
	\tilde{G}_{\alpha,\rho}(u, t)$, $|u| \leq 1$, solves the equation
	\begin{equation*}
		\left({\mathcal D}^{2\alpha,\rho}_t + {\mathcal D}^{\alpha,\rho}_t\right)  \tilde G_{\alpha,\rho}(u, t) = -\sum_{j=1}^{k}\lambda_j(1-u^j)\tilde G_{\alpha,\rho}(u, t), 
	\end{equation*}
	with $\tilde G_{\alpha,\rho}(u, 0) = 1$; 
  the generalized fractional derivatives in the C-D sense are defined in \eqref{exfDt}.
\end{theorem}
\begin{proof}
Equation for probabilities \eqref{fM2eqt} is derived by following the same lines as in the proof of Theorem \ref{fPequ} with the use of the  equation for the density of the inverse process \eqref{temp3} and the governing equation for the generalized  counting process \eqref{pGCP}. The governing equation for the probability generating function folows from Theorem \ref{Th5}.
\end{proof}

\begin{remark} For the probabilities 
	 ${p}^{\alpha,\rho}_n(t)$ which solve the equation \eqref{fPtemp} it possible to write an integral representations involving the multivariate Mittag-Leffler functions. Namely, similar to the proofs of Theorems \ref{fP2}, \ref{fP3}, we can consider the Laplace transform of equation \eqref{fPtemp}  from which the following expression for the Laplace transform of ${p}^{\alpha,\rho}_n(t)$ can be derived:
	\begin{equation*}
		\hat{p}^{\alpha,\rho}_n(t) = \frac{(s+\rho)^{2\alpha} - \rho^{2\alpha} + (s+\rho)^\alpha - \rho^\alpha}{s \left[ (s+\rho)^{2\alpha} - \rho^{2\alpha} + (s+\rho)^\alpha - \rho^\alpha + \lambda \right]^{n+1}} :=H(s).
	\end{equation*}
	
	To find the inverse Laplace transform we can use the shift property of the Laplace transform and the integration property associated with the $1/s$ factor, so that
	\begin{equation*}
		\mathcal{L}^{-1}\{H(s)\}(t) = \int_0^t e^{-\rho \tau} f(\tau) d\tau
	\end{equation*}
	where $f(t) = \mathcal{L}^{-1}\{F(s)\}(t)$, and $F(s)$ is defined as:
	\begin{equation*}
		F(s) = \frac{s^{2\alpha} + s^\alpha - C}{\left( s^{2\alpha} + s^\alpha - C + \lambda \right)^{n+1}}
	\end{equation*}
	with the constant $C = \rho^{2\alpha} + \rho^\alpha$.
		Next we  can decompose $F(s)$ into two  terms:
	\begin{equation*}
		F(s) = \frac{1}{(s^\alpha - r_1)^n (s^\alpha - r_2)^n} - \frac{\lambda}{(s^\alpha - r_1)^{n+1} (s^\alpha - r_2)^{n+1}},
	\end{equation*}
where	$r_{1,2} = -\frac{1}{2} \pm \frac{1}{2}\sqrt{1 + 4(C - \lambda)}$.

	This form allows for a direct application of the formula \eqref{multMLfLapl} which represents the Laplace transform for the multivariate generalized Mittag-Leffler function. 	
	By matching parameters for our two terms as 
	 $\gamma_1 = \gamma_2 = n$,   $\delta = 2\alpha n$ and 
		$\gamma_1 = \gamma_2 = n+1$,  $\delta = 2\alpha(n+1)$, correspondingly, 
	we can present the inverse transform $f(t)$:
	\begin{equation*}
		f(t) = t^{2\alpha n - 1} E_{\alpha, 2\alpha n}^{(n, n)}(r_1 t^\alpha, r_2 t^\alpha) - \lambda t^{2\alpha(n+1) - 1} E_{\alpha, 2\alpha(n+1)}^{(n+1, n+1)}(r_1 t^\alpha, r_2 t^\alpha),
	\end{equation*}
	where $E_{\nu,\delta}^{(\gamma_1, \gamma_2)}$ is the two-variable generalized Mittag-Leffler function.
	Summarizing all the above, the  inverse Laplace transform of $H(s)$ is:
	\begin{align*}
		\mathcal{L}^{-1}\{H(s)\}(t) &= \int_0^t e^{-\rho \tau} \bigg[ \tau^{2\alpha n - 1} E_{\alpha, 2\alpha n}^{(n, n)}(r_1 \tau^\alpha, r_2 \tau^\alpha) 
		- \lambda \tau^{2\alpha(n+1) - 1} E_{\alpha, 2\alpha(n+1)}^{(n+1, n+1)}(r_1 \tau^\alpha, r_2 \tau^\alpha) \bigg] d\tau,
	\end{align*}
	which gives the expression for ${p}^{\alpha,\rho}_n(t)$.
\end{remark}

\section{An application to risk theory}

As one of possible applications, GCP can serve to generalize classical risk models as discussed, for example, in \cite{KK2022b}, where, in particular,  the governing equations for ruin probability were derived together with the closed expression for ruin probability with zero initial capital. 

In this section we provide  an expression, in terms of the Mittag-Leffler functions, for the ruin probability when the initial capital $u>0$.

Consider the following risk model with the GCP as the counting process:
\begin{equation*}
	X(t) = ct - \sum_{i=1}^{M(t)} Z_i, \quad t \ge 0,
\end{equation*}
where $c > 0$ denotes the constant premium rate, $\{Z_i\}_{i \ge 1}$ is the sequence of positive iid random variables representing the individual claim sizes, which are independent of the GCP $M(t)$. 

Let $F$ be the distribution function of $Z_i$ and  $\mu = \mathbb{E}(Z_i)$. The relative safety loading factor $\eta$ for this risk model is given by:
\begin{equation*}
	\eta = \frac{\mathbb{E}(X(t))}{\mathbb{E} \left( \sum_{j=1}^{M(t)} Z_j \right)} = \frac{c}{\mu \sum_{j=1}^k j \lambda_j} - 1.
\end{equation*}
Hence, the condition $c > \mu \sum_{j=1}^k j \lambda_j$ must hold for the safety loading factor to be positive. 

Let $u \ge 0$ denote the initial capital and $\{U(t)\}_{t \ge 0}$ be the surplus process
$
U(t) = u + X(t).
$

Let $\tau$ denote time to ruin: $\tau=\inf\{t>0:U(t)<0\}$, the ruin probability is given by $\psi(u)=\Prob\{\tau<\infty\}$. Correspondingly, the non-ruin or survival probability is
$$ \phi(u)=1-\psi(u), u\ge 0.$$

The integro-differential equation for the survival probability for the model with the GCP is given as (see, \cite{KK2022b}):
\begin{equation}\label{phi_eq}
	\frac{d}{du}\phi(u) = \frac{\Lambda}{c}\Bigl( \phi(u) - \frac{1}{c} \sum_{j=1}^k \lambda_j \int_0^u \phi(u-z)dH(z)\Bigr), \quad u > 0,
\end{equation}
where 
$$
H(x) =\frac{1}{\Lambda}\sum_{j=1}^k \lambda_j F^{\ast j}(x)
$$
is the mixture distribution with components being $j$-fold convolutions of the distribution $F$, which give the distributions of the aggregated claims $Z_1+\ldots+Z_j$.

The non-ruin probability for the case of zero initial capital was obtained in \cite{KK2022b}:
$$\phi(0)=1-\frac{\mu}{c}\sum_{j=1}^k j\,\lambda_j.
$$

We present the expression for $\phi(u)$, $u>0$, 
for the case of gamma distributed claim sizes, i.e., with the density
\begin{equation*}
	f_Z(x) = \frac{\alpha^r}{\Gamma(r)} x^{r-1} e^{-\alpha x}, \quad x > 0,
\end{equation*}
where $r > 0$ is the shape parameter, and $\alpha > 0$ is the scale parameter. For this case the generalization of the result from \cite{nonruin} can be stated.

\begin{theorem}
	Assume the individual claim sizes $Z$ follow the  gamma distribution with shape parameter $r > 0$ and scale parameter $\alpha > 0$. Then  the non-ruin probability is given by:
	\begin{equation}\label{phiu}
		\phi(u) = \phi(0)+ e^{-\alpha u} \phi(0) \left\{ e^{\alpha u} \ast \sum_{n=1}^{\infty} \left( \frac{\Lambda}{c} \right)^n \left[ e^{\alpha u} - \frac{1}{\Lambda} \sum_{j=1}^k \lambda_j (\alpha u)^{jr} E_{1,1+jr}(\alpha u) \right]^{\ast n} \right\},
	\end{equation}
	for any $u > 0$, where $\ast$ denotes the convolution operator, $f^{\ast n}$ denotes the $n$-fold convolution power, and $E_{\alpha,\beta}(z)$ is the two-parameter Mittag-Leffler function.
\end{theorem}

\begin{proof}
From the equation \eqref{phi_eq} we 
 conclude  that the Laplace transform of the non-ruin probability
$
	\hat{\phi}(s) = \int_0^\infty e^{-su} \phi(u) du
$, $u>0$,
is given by
\begin{equation*}
	\hat{\phi}(s) = \frac{c\phi(0)}{cs - \Lambda + \sum_{j=1}^k \lambda_j M_{Y_j}(-s)} = \frac{c\phi(0)}{cs - \Lambda + \sum_{j=1}^k \lambda_j \left(\frac{\alpha}{s+\alpha}\right)^{jr}}\,\,, 
\end{equation*}
where $M_{Y_j}(s)$ is the moment generating function of the aggregate claims $Y_j=Z_1+\ldots+Z_j$.

	We can rewrite the above expression in the following form:
	\begin{align*}
		\hat{\phi}(s) &= \frac{\phi(0)}{s} \frac{1}{1 - \frac{\Lambda}{c} \left( \frac{1}{s} -\frac{1}{s} \frac{1}{\Lambda} \sum_{j=1}^k \lambda_j \left(\frac{\alpha}{s+\alpha}\right)^{jr}\right)} 
		= \frac{\phi(0)}{s} \sum_{n=0}^{\infty} \left( \frac{\Lambda}{c} \right)^n \left[ \frac{1}{s} - \frac{1}{s} \frac{1}{\Lambda} \sum_{j=1}^k \lambda_j \left(\frac{\alpha}{s+\alpha}\right)^{jr} \right]^n.
	\end{align*}
	
	For $s > \alpha$, we shift the argument to obtain $\hat{\phi}(s - \alpha)$:
	\begin{equation*}
		\hat{\phi}(s - \alpha) = \frac{\phi(0)}{s - \alpha} \sum_{n=0}^{\infty} \left( \frac{\Lambda}{c} \right)^n \left[ \frac{1}{s - \alpha} - \frac{1}{\Lambda} \sum_{j=1}^k \lambda_j \frac{\alpha^{jr}}{(s - \alpha)s^{jr}} \right]^n.
	\end{equation*}

	The inverse Laplace transform of the expression in the square brackets is:
	\begin{align*}
		\mathcal{L}^{-1} \left\{ \frac{1}{\Lambda} \sum_{j=1}^k \lambda_j \left( \frac{1}{s - \alpha} - \frac{\alpha^{jr}}{(s - \alpha)s^{jr}} \right) \right\} 
		= e^{\alpha u} - \frac{1}{\Lambda} \sum_{j=1}^k \lambda_j (\alpha u)^{jr} E_{1,1+jr}(\alpha u).
	\end{align*}
	Applying the shifting theorem back to the time domain yields the desired convolution series for $\phi(u)$.
\end{proof}

\begin{remark}
Note that for an integer parameter $r$  we can write
$$
E_{1,1+jr}(\alpha u)=\frac{1}{(\alpha u)^{jr}}\left( e^{\alpha u}-\sum_{k=0}^{jr-1}\frac{(\alpha u)^k}{k!}\right),
$$
therefore, the formula \eqref{phiu} simplifies to the following form:
\begin{equation*}
		\phi(u) = \phi(0)+ e^{-\alpha u} \phi(0) \left\{ e^{\alpha u} \ast \sum_{n=1}^{\infty} \left( \frac{\Lambda}{c} \right)^n \left[  \frac{1}{\Lambda} \sum_{j=1}^k \lambda_j \sum_{l=0}^{rj-1}\frac{(\alpha u)^l}{l!} \right]^{\ast n} \right\}.
	\end{equation*}
\end{remark}


\end{document}